\documentclass{amsart}
\usepackage{amsfonts}
\usepackage{graphicx}
\usepackage{amscd}
\usepackage{amsmath}
\usepackage{amssymb}
\usepackage[dvipsnames]{xcolor}

\makeatletter
\@namedef{subjclassname@2010}{%
\textup{2010} Mathematics Subject Classification.}
\makeatother

\theoremstyle{plain}

\newtheorem{theorem}{\bf Theorem}[section]
\newtheorem{corollary}[theorem]{\bf Corollary}
\newtheorem{definition}[theorem]{\bf Definition}

\newtheorem{lemma}[theorem]{\bf Lemma}

\newtheorem{proposition}[theorem]{\bf Proposition}
\newtheorem{remark}[theorem]{Remark}

\numberwithin{equation}{section}

\title[q solitons] {Hamilton's identity and rigidity of complete gradient solitons}

\author[A.W. Cunha, A. N. Silva Jr. and Will Wylie]{Antonio W. Cunha$^{1}$; Antonio N. Silva Jr$^2$. and William Wylie$^3$}

\address{$^{1}$ Department of Mathematics, Universidade Federal do Piau\'{\i}, 64049-550 Teresina, Piau\'i, Brazil. Currently: Dept. of Math., Syracuse University, NY, 13244,  as a Scholar visitor.}

\address{$^{2}$ Universidade Estadual do Piauí, Campus Floriano, 64808-080, Floriano, Piauí, Brazil.}

\address{$^3$ Carnegie Building, Dept. of Math, Syracuse University, Syracuse, NY, 13244.}

\email{$^1$wilsoncunha@ufpi.edu.br}
\email{$^2$antoniojunior@frn.uespi.br}
\email{$^3$wwylie@syr.edu}



\begin{document}

\begin{abstract}
In this work, we study gradient solitons to general geometric flows.  Our approach is to understand what assumptions need to be made about a flow in order to extend results about Ricci solitons.  In this direction,  we identify an identity, first exploited in the pioneering work of Richard  Hamilton in the case of Ricci solitons, which we call Hamilton's identity.  We show that a version of this identity for an arbitrary geometric flow allows one to recover results about rigidity, the growth of the potential function, volume growth and the Omori-Yau maximum principle that have been proven for gradient Ricci solitons. 
\end{abstract}

\maketitle

\section{Introduction}

\noindent

In recent years, many intrinsic geometric flows have been introduced. The most well studied is the Ricci flow, introduced by Hamilton \cite{Hamilton}, which was fundamental in solving the Poincaré Conjecture due to Perelman.  The  self-similar solutions to the Ricci flow are called Ricci solitons. These are Riemannian manifolds $(M, g)$ that satisfy the equation
$${\rm Ric}+ \frac{1}{2}\mathcal{L}_X g=\lambda g,$$
for some vector field $X$, where $\mathcal{L}_X g$ denotes the Lie derivative, ${\rm Ric}$ is the Ricci curvature of $M$ and $\lambda\in\mathbb{R}$. When $X=\nabla f$ for some function $f:M\to\mathbb{R}$ the soliton is called a \emph{gradient Ricci soliton}.  

Much work has been done on gradient Ricci solitons. In fact, in $3$-dimension it was proved that {\em any 3-dimensional complete noncompact non-flat shrinking gradient soliton is necessarily
the round cylinder $\mathbb{S}^2\times\mathbb{R}$ or one of its $\mathbb{Z}_2$ quotients} (see \cite{Cao-Chen-Zhuo},\cite{Wall}). 

In this paper we are interested in to what extent the theory of gradient Ricci solitons can be extended to other geometric flows.  For example, solitons for the Ricci Bourguignons flow, the ambient obstruction flow, and the $G_2$-closed Laplacian flow have received attention recently (\cite{Catino}, \cite{LY}, \cite{Lopez}). 

The study of Ricci solitons involves an interesting interplay of the geometry and curvature of the manifold and the function $f$ and its derivatives.  For gradient solitons to other flows, the function $f$ remains, but the Ricci tensor is replaced with another tensor, so in this paper we will focus on the potential function $f$.  

In order to formalize this idea,  Griffin \cite{Griffin} introduced the notion of a $q$-soliton.   Let $q$ be a $(0,2)$-tensor we call a flow of for form
\begin{equation}\label{q flow}
\left\{%
\begin{array}{ll}
\frac{\partial}{\partial t}g=q & \\
g(0)=h,
\end{array}%
\right.
\end{equation}
the {\emph{q -flow}}. Assuming scale invariance of $q$, the self-similar solutions  are characterized by solving the following equation
\begin{equation}\label{q soliton}
\frac{1}{2}\mathcal{L}_Xg=\lambda g+\frac{1}{2}q,
\end{equation}
where $\lambda$ is a constant. We call a solution to (\ref{q soliton}) a \emph{q-soliton}.  If $X=\nabla f$ for some smooth function $f$ on $M$, the $q$-soliton is said to be a {\em gradient $q$-soliton} and equation \eqref{q soliton} becomes
\begin{equation}\label{grad q soliton}
{\rm Hess} f=\lambda g+\frac{1}{2}q,
\end{equation}
where ${\rm Hess} f$ is the Hessian of  potential function $f$. If we  let $Q$ denote the $(1,1)$-tensor dual to $q$, we can also write the equation as
$$\nabla_X \nabla f = \lambda X + \frac{1}{2} Q(X).$$
We say that the $q$-soliton is {\em expanding, steady} and {\em shrinking} if $\lambda < 0$, $= 0$ or $>0$, respectively. It is said to be {\em stationary} if $X$ is a Killing vector field or the potential function $f$ is constant in the case of a gradient $q$-soliton. A $q$-soliton is said to be {\em $q$-flat} if $q = 0$. We will call a $q$-soliton {\em trivial} if $M$ is stationary and $q$-flat.

Griffin \cite{Griffin} proved that for compact $q$-solitons if $q$ is  divergence-free and trace-free  then any compact $q$-soliton is $q$-flat. This generalizes a previous result for Ricci solitons that any compact Ricci soliton with constant scalar curvature is Einstein (see \cite{WP}). More recently, the first author, together with Griffin \cite{Cunha and Griffin}, considered the complete case and obtained several $q$-flatness results by assuming restrictions on Ricci curvature, parabolicity, and either $L^\infty$ or $L^p$-regularity in the potential function or the vector field in the non-gradient case.

In this paper we are interested in other assumptions about $q$ that may be needed to generalize results about Ricci solitons to $q$-solitons.   One of the most well-known and important identities for gradient Ricci solitons is the identity due to Hamilton \cite{Hamilton} that the quantity  
$$ R +  |\nabla f|^2 - 2\lambda f$$ is constant, where $R$ denotes the scalar curvature. This turns out to follow from the tensor equation, ${\rm Ric}(\nabla f,\cdot)=\frac{1}{2}d R( \cdot) $, which is a consequence of the contracted second Bianchi identity and the Ricci soliton equation.  To generalize this to other flows, we will say that a gradient  $q$-soliton satisfies a Hamilton type identity if 
 \begin{equation}\label{equality assumption}
     Q(\nabla f) =\frac{1}{2}\nabla  \operatorname{tr}(q).
 \end{equation}   We will see below that this tensor equation is equivalent to the quantity
 
$$|\nabla f|^2 - \frac{1}{2}\operatorname{tr}(q)-2\lambda f
$$ being constant.  Which, tracing the $q$-soliton equation, we can also write in terms of an equation only involving $f$,  

$$ \Delta f + |\nabla f|^2 - 2 \lambda f = \text{constant}  $$(see Proposition \ref{Hamilton identity}).

 We are interested in what features gradient solitons that satisfy a Hamilton identity share with gradient Ricci solitons.  Our first result is the well-known result of Hamilton that non-stationary  compact gradient Ricci solitons must be shrinking.

 \begin{theorem}\label{thm01} If $(M,g,f)$ is a compact gradient $q$-soliton  with $\lambda \leq 0$ that satisfies (\ref{equality assumption}) then $f$ is constant. 
 \end{theorem}

The third author and Petersen \cite{WP} have also proved that for a non-trivial compact shrinking gradient Ricci soliton  
$$\int_M{\rm Ric}(\nabla f,\nabla f)>0.$$

 Our next result is a generalization of this result  for gradient $q$-solitons satisfying \eqref{equality assumption}.
 
\begin{theorem}\label{thm03}
Any compact non-stationary gradient $q$-soliton satisfying \eqref{equality assumption} must satisfy $$\int_M  \mathrm{div}(q)(\nabla f)< 0.$$
\end{theorem}

 Our next set of results are generalizations of results about complete non-compact gradient Ricci solitons with constant scalar curvature.  In the Ricci soliton case, constant scalar curvature is equivalent to $q$ having constant trace and, by the Bianchi identity, $q$ being divergence free.

\begin{theorem}\label{thm02}
Let $(M,g,f)$ be a complete gradient $q$-soliton satisfying \eqref{equality assumption}. If one of the following holds:
\begin{enumerate}
    \item [i)] either $q$ is trace free and divergence free, or
    \item [ii)] $q$ is trace free and non-negative Ricci curvature, or
    \item [iii)] $\lambda=0$, $q$ is constant trace and non-negative Ricci curvature, or
    \item[iv)] $\lambda<0$, $div(q)=0$, non-positive trace and either $M$ is parabolic or non-negative Ricci curvature,
\end{enumerate}
then $M$ must be $q$-flat.
\end{theorem}
\begin{remark}
{\em 
We note that, at least in case i), the assumption \eqref{equality assumption} is necessary for $q$-flatness.  For example, Ho \cite{Ho} shows that the product metric $\mathbb{R}^2 \times N^2$ where $N$ is a surface of constant curvature is a shrinking gradient Bach soliton and Griffin constructs a homogeneous complete expanding gradient Bach soliton on $\mathbb{R} \times SU(2)$.   These examples have $q$ trace free and divergence free, but are  not $q$-flat.  This also shows that gradient Bach solitons do not, in general, satisfy (\ref{equality assumption}). }
\end{remark}

To state the next result, we say that a gradient $q$-soliton $M$ is {\em rigid} if its universal cover is isometric to a product metric   $M=N\times \mathbb{R}^k$,  and $f$ is a quadratic function on the $\mathbb{R}^k$ factor. (See Definition \ref{rigid-defn} for the precise definition). 

In \cite[Theorem 1.2]{WP} the third author and Petersen proved that a gradient Ricci soliton is rigid if and only if it has constant scalar curvature and is {\em radially flat}, that is,
$${\rm sec}(X,\nabla f)=0.$$
We obtain the following version of this result for $q$-solitons.

\begin{theorem}\label{thm04}
Let $(M,g,f)$ be a complete gradient $q$-soliton.  $(M,g,f)$ is rigid if and only if  $\mathrm{tr}(q)$ is constant, ${\rm sec}(X,\nabla f)=0$,  and there is a constant $c$ such that $Q(\nabla f) = c \nabla f$. 
\end{theorem}

\begin{remark}
{\em 
If a gradient $q$-soliton has constant trace and satisfies (\ref{equality assumption}) then clearly $Q(\nabla f) = 0$, so Theorem \ref{thm04} implies that a constant trace, radially flat soliton satisfying (\ref{equality assumption}) is rigid.  On the other hand, the Ho and Griffin examples of Bach solitons are rigid and have $c \neq 0$. }
\end{remark}

Moving beyond constant scalar curvature,  Cao and Zhuo \cite{Cao-Zhou} prove that for a  complete gradient shrinking Ricci soliton,  its potential function satisfies
$$\frac{1}{4}\left(r(x)-c_1 \right)^{2}\leq f(x)\leq \frac{1}{4}\left(r(x)+c_2 \right)^{2}.$$
The key tool used in their proof is Hamilton's identity and the fact that gradient shrinking Ricci solitons have nonnegative scalar curvature \cite{Chen}.

For any Riemannian manifold $M$ consider the function $F$ given by 
$$F:=\frac{1}{2}|\nabla f|^2-\lambda f,$$
for some smooth function $f:M\to\mathbb{R}$ and $\lambda>0$. Note that by Hamilton's identity and the fact that shrinking Ricci solitons have non-negative scalar curvature, $F$ is always bounded above for a gradient shrinking Ricci soliton. We observe that $F$ bounded above is enough to show a quadratic upper bound on the function $f$. 

\begin{theorem}\label{thm05}
Let $(M,g)$ be a complete non-compact Riemannian manifold. If $F$ is bounded from above, then the function $f$ satisfies
 \begin{equation}
      f(x)\leq \frac{1}{4}\left(r(x)+c_1 \right)^{2}.
 \end{equation}
 Here $r(x)=d(x_0,x)$ is the distance function of some fixed $x_0 \in M$.
\end{theorem}

\begin{remark}\label{Rmk01}
{\em 
On the other hand, bounded $F$ does not sufficient for a lower bound on the growth of $f$.  The lower bound on the potential function depends much more heavily on the properties of the Ricci tensor, in particular its appearance in the second variation of energy formula.  If
$F$ is bounded from above, and we also assume that, for any minimizing normal geodesic $\gamma(s)$, $0\leq s\leq s_0$, the following inequality holds 
\begin{equation}\label{Integral ineq}
\int_{0}^{s_0}\phi^2\left(-\frac{1}{2}q(\gamma',\gamma')\right)ds\leq (n-1)\int_{0}^{s_0}|\phi^{\prime}(s)|^2 ds,
    \end{equation}
 for every nonnegative function $\phi$, then we have 
 \begin{equation}\label{bound from below}
    f(x)\geq \frac{1}{4}\left(r(x)-c_2 \right)^{2},
 \end{equation}
 as we will prove in the Proposition \ref{Prop03}.}
\end{remark}

Cao and Zhu also prove a volume growth estimate for shrinking gradient Ricci solitions. The upper bound on volume turns out the follow from \eqref{equality assumption} and an upper bound on $F$ while the lower bound follows from \eqref{Integral ineq}. In the upper bound case, we prove the following:

\begin{theorem}\label{thm06}
Let $\left(M,g,f\right)$ be a complete non-compact shrinking gradient $q$-soliton satisfying \eqref{equality assumption}. If $F$ is bounded from above then there is a positive constant $C_1>0$ such that \begin{equation}
\operatorname{Vol}\left(B_{x_0}(r)\right) \leq C_1 r^n
\end{equation}
for $r>0$ sufficiently large.
\end{theorem}

 We conclude this work by proving the validity of the Omori-Yau maximum principle. For this, let us recall that a Riemannian manifold $(M,g)$ satisfies the {\em Omori–Yau maximum principle} if given any
function $u\in C^2(M)$ such that $u^\ast=\sup_M u<+\infty$, there exists a sequence $(x_k)\subset M$ of points on $M$ satisfying
\begin{equation}\label{eqOY}i)\,\,\, u(x_k)> u^{\ast}-1/k, \,\,\,\;ii)\,\, \,\vert \nabla u \vert (x_k)<1/k,\,\,\,\;iii)\,\,\, \triangle u (x_k) < 1/k\,. \end{equation}
It was shown by H. Omori  \cite{omori} that any open Riemannian manifold  $M$ with sectional curvatures bounded below  and $u\in C^{2}(M)$ bounded from above, satisfies the Omori–Yau maximum principle. After, Cheng and Yau in \cite{cheng-yau-math-ann-77}, \cite{yau}, have shown its validity for any  Riemannian manifolds with Ricci curvature bounded from below. More recently, Pigola, Rigoli and Setti have shown that the validity of the Omori-Yau maximum principle by
 the existence of a proper  smooth function $\theta \in C^{2}(M)$ satisfying certain natural geometric properties, see \cite[Theorem 1.9]{PRS-Memoirs}. 
 
It's important to point out that the condition ii) is not necessary in many geometric applications. So,  Pigola, Rigoli and Setti  \cite{prs-pams}, \cite{PRS-Memoirs}, introduced the concept of weak maximum principle by define that the {\em weak maximum
principle at infinity} holds on a Riemannian manifold $M$ if for any function $u\in C^{2}(M)$ with $u^{\ast}< \infty$ there exists a sequence $\{x_k\}\subset M$
satisfying the conditions \[a)\,\, u(x_k)> u^{\ast}-1/k, \,\,\,\;\,\,\,\,\,\,\,\,b)\,\, \Delta u (x_k) < 1/k. \]
In this sense they proved that the weak maximum principle for
 the Laplacian is equivalent to the stochastic completeness of a Riemannian manifold (see \cite{prs-pams}). For more details about stochastic completeness the reader can find in \cite{grigoryan}. 
 
 With this notion of maximum principle, in 2011, Pigola, Rimoldi and Setti \cite{PRS} have studied its validity for gradient Ricci solitons. They proved that a geodesically complete Ricci soliton which is either
shrinking, steady or expanding, satisfies the weak maximum principle at infinity for the $f$-Laplacian, where it is known that 
$$\Delta_fu=\Delta u-\langle\nabla f,\nabla u\rangle.$$

 Following this new perspective, in the same year García-Río and Fernández-López \cite{GR} have studied the validity of the Omori-Yau maximum principle on gradient shrinking Ricci solitons both for the Laplacian
and the $f$-Laplacian. They showed that any $n$-dimensional complete noncompact gradient shrinking Ricci soliton satisfies the Omori–Yau maximum principle. Hence, a natural task is to allow conditions to have the validity of the Omori-Yau maximum principle for gradient $q$-solitons in view they are generalizations of gradient Ricci solitons. Therefore, we give an answer to gradient $q$-solitons satisfying \eqref{equality assumption} in which generalizes the Theorem 2.1 of \cite{GR} to gradient $q$-solitons.

\begin{theorem}\label{thm07}
    Let $(M,g,f)$ be a complete non-compact, shrinking gradient $q$-soliton satisfying \eqref{equality assumption}. If $F$ is bounded from above and \eqref{Integral ineq} holds, then the Omori-Yau maximum principle holds on $M$ both for Laplacian and $f$-Laplacian.
\end{theorem}

Since the validity of the Omori-Yau maximum principle implies stochastic completeness, i.e., it holds the weak maximum principle on $M$, then as a consequence of  Pigola, Rigoli and Setti \cite{PRS-Memoirs} we obtain

\begin{corollary}
    Let $(M,g, f)$ a complete non-compact, shrinking gradiente $q$-soliton satisfying \eqref{equality assumption}. If $F$ is bounded from below and \eqref{Integral ineq} holds, then the following statements are equivalent:
    \begin{enumerate}

        \item [i)] $M$ is stochastically complete;

        \item [ii)] For every $\lambda >0$, the only non-negative, bounded smooth functions solution $u$ of $\Delta u \geq \lambda u$ on $M$ is $u\equiv0$;
        
        \item [iii)] For every $\lambda>0$, the only non-negative, bounded smooth solution $u$ of $\Delta u=\lambda u$ on $M$ is $u\equiv0$;

        \item [iv)] For every $T>0$, the only bounded solution on $M\times (0,T)$ of Cauchy problem \begin{equation*}
        \left\lbrace \begin{matrix}
            \frac{\partial u}{\partial t} & =\frac{1}{2}\Delta u \\
            u|_{t=0^+} & =0 \ \textit{in the $L^{1}_{loc}(M)$ sense},
        \end{matrix} \right.
        \end{equation*} 
is $u\equiv 0$.
    \end{enumerate}
\end{corollary}
In particular the above corollary implies Corollary 2.4 of \cite{GR} if we take $q=-2{\rm Ric}$.

\noindent
\section{Rigidity Characterization}

Let $(M,g)$ be a Riemannian manifold and $f$ a smooth function on $M$.  Let $\Lambda \in \mathbb{R}$, and  consider the function 
\[ F_{\Lambda} = \frac{1}{2}|\nabla f|^2 -  \Lambda f \]

Our motivation for considering $F$ comes from the $q$-soliton equation, but in this section we will focus on the function $F$.   In particular we are after a characterization of rigid spaces in the following sense: 

\begin{definition} \label{rigid-defn}
A Riemannian manifold, $(M,g)$ equipped with a smooth function $f$ is rigid if 
\[ M = (N \times \mathbb{R}^k)/{\Gamma}\]
where $N \times \mathbb{R}^k$ is the Riemannian product of a Riemannian manifold $N$ and a flat metric on $\mathbb{R}^k$, $\Gamma$ is a discrete group of isometries acting via $O(n)$ on the Euclidean factor and freely on the $N$ factor, and $f$ is a function on the $\mathbb{R}^k$ factor of the form $f(X) = \frac{\Lambda}{2}|X|^2 +L(X) + b$ where $L:\mathbb{R}^k \to \mathbb{R}$ is a linear map and $\Lambda$ and $b$ are constants.  
\end{definition}

Note that, equivalently, a rigid space is a Riemannian manifold $(M,g)$ with function $f$ such that, when lifting the metric and function to the universal cover one gets a product metric with a Euclidean factor such that $f$ is a quadratic function on the Euclidean factor. 

These rigid examples come up as examples when studying solitons as $\mathrm{Hess} f|_{\mathbb{R}^k} = \Lambda g_{\mathbb{R}^k}$ on the Euclidean factor.  Since Euclidean space is isotropic,  $\Lambda g_{\mathbb{R}^k}$ is the only isometry-invariant tensor on Euclidean space, so one expects these rigid examples to appear often for flows where $q$ is isometry-invariant.  

In the Ricci soliton case, $\mathrm{Hess} f = \lambda g - \mathrm{Ric}$,  a rigid example must have $N$ a $\lambda$- Einstein metric and $\Lambda = \lambda$.  However, this need not be the case in general.  For example,  the rigid examples of Bach solitons found in \cite{Griffin, Ho} have $\Lambda \neq \lambda$.  Specifically,  for the rigid product $N^2 \times \mathbb{R}^2$ where $N^2$ is a surface of constant curvature, the Bach tensor is of the form $B = -c^2 g_{N^2} + c^2 g_{\mathbb{R}^2}$ where $c$ is a constant depending on the curvature, so one has a rigid soliton with $\lambda = \frac{1}{2} c^2$ but $\Lambda = c^2$. 

The purpose of this section is to prove the following characterization of rigid spaces. 

\begin{theorem}  \label{thm-rigid}
Let $(M,g)$ be a complete Riemannian manifold and $f:M \to \mathbb{R}$ a smooth function.  Let $\Lambda$ be a constant, and $F_{\Lambda} = \frac{1}{2}|\nabla f|^2 - \Lambda f$.  Then $(M,g,f)$ is rigid if and only if 
$F_{\Lambda}$ is constant, $\Delta f$ is constant  and $R(X, \nabla f) \nabla f = 0$. 
\end{theorem}

Before we prove Theorem \ref{thm-rigid}, we point out to the application to the $q$-soliton equation. 

\begin{proposition}\label{Prop Q}
    Let $(M,g, f)$ be a gradient $q$-soliton, $ \mathrm{Hess} f = \lambda g + \frac{1}{2} q $ and let $F_{\Lambda} = \frac{1}{2}|\nabla f|^2 -  \Lambda f$.  Then 
    \begin{enumerate}
        \item $\Delta f$ is constant if and only if $\mathrm{tr}(q)$ is constant.
        \item $F_{\Lambda}$ is constant if and only if $Q(\nabla f) = 2(\Lambda - \lambda) \nabla f$. 
         \end{enumerate}
\end{proposition}

\begin{proof}
(1) is obvious from tracing the $q$-soliton equation.  For (2), consider the $(1,1)$ version of the soliton equation
\[ \nabla_X \nabla f = \lambda X + \frac{1}{2} Q(X).\]
Then
\begin{eqnarray*}
    \nabla F_{\Lambda} & =& \nabla |\nabla f|^2 - 2 \Lambda \nabla f \\&=&  2 \nabla_{\nabla f} \nabla f - 2 \Lambda \nabla f \\
    &=& 2 \lambda \nabla f + Q(\nabla f)  - 2 \Lambda \nabla f \\
    &=& Q(\nabla f) - 2 (\Lambda - \lambda)\nabla f
\end{eqnarray*}
which gives (2). 
\end{proof}

\begin{proof}[Proof of Theorem \ref{thm04}]
    Proposition \ref{Prop Q} implies that there is a $\Lambda$ such that  $F_{\Lambda}$ and $\Delta f$ are constant. Since $\sec(X,\nabla f)=0$ is equivalent to $R(X,\nabla f)\nabla f=0$, the proof follows from Theorem \ref{thm-rigid}.
\end{proof}

Now we turn our attention to proving the theorem, we first consider the case $\Lambda =0$ separately.  In this case, we can replace the sectional curvature assumption with a weaker 
Ricci curvature assumption. 

\begin{proposition} \label{prop-rigid-0}
    Let $(M,g)$ be a complete Riemannian manifold and $f$ a non-constant smooth function with $|\nabla f|$ constant and $\mathrm{Ric}(\nabla f, \nabla f) \geq 0$, then $M$ is isometric to $N \times \mathbb{R}$ and $f =ar+b$ where $r$ is the parameter in the $r$ direction. 
\end{proposition}

\begin{proof}
    From the Bochner formula, 
    \[ \frac{1}{2} \Delta|\nabla f|^2 = \mathrm{Ric}(\nabla f, \nabla f) + |\mathrm{Hess} f|^2 + g(\nabla \Delta f, \nabla f). \]
    So our assumptions along with the Cauchy-Schwarz inequality gives that 
    \[ g(\nabla \Delta f, \nabla f) \leq -|\mathrm{Hess} f|^2 \leq - \frac{(\Delta f)^2}{n}. \]
    Let $\gamma(t)$ be an integral curve of $\nabla f$. Since $|\nabla f|$ is constant, $\gamma$ is a geodesic so by completeness, exists for all time $(-\infty, \infty)$.  Let $\phi(t) = \Delta f(\gamma(t))$. Then the equation above gives that $\phi$ satisfies the differential inquality 
    \begin{align} \label{eqn-Ricat}
      \phi'(t) \leq \frac{\phi^2}{n}.  
    \end{align} 
    The Ricatti comparison then allows us to bound $\phi$ in terms of a solution,  $\overline{\phi}(t)$ to the equation
    \[ \overline{\phi}' = \frac{\overline{\phi}^2}{n}. \]
     Namely, if 
    \[ G(t) = e^{ \int_{t_0}^{t} -\frac{1}{n}(\phi(s) + \overline{\phi}(s)) ds } \] then a simple calculation shows that the function $\frac{d}{dt} G(\phi-\overline{\phi})\leq 0$.  

    Since $G$ is a positive function, this implies that if $\phi(t_0) = \overline{\phi}(t_0) $ for some $t_0$ then $\phi(t) \leq \overline{\phi}(t)$ for all $t \geq t_0$ and $\phi(t) \geq \overline{\phi}(t)$ for all $0<t \leq t_0$.  The solutions $\overline{\phi}$ can be computed explicitly as 
    \[ \overline{\phi}(t) = 0 \quad \text{or} \quad \overline{\phi}(t)  = \frac{1}{a+ \frac{t}{n}} \]
    for a constant $a$.  In particular the functions $\overline{\phi}$ have the property that if $\overline{\phi}(t_0)<0$ then there is $t_1>t_0$ such that $\overline{\phi}(t) \to -\infty$ as $t\to t_1^-$, $t>t_0$.  The comparison principle then implies that if $\phi(t)<0$ at some point then $\phi$ must also go to $-\infty$ in finite time, a contradiction. Similarly, the functions $\overline{\phi}$ have the property that if $\overline{\phi}(t_0) > 0$, then there is $t_1< t_0$ such that $\overline{\phi}(t) \to \infty$ as $t \to t^+$.  The comparison principle then says that if $\phi(t)>0$ then $\phi$ would have to also go to $+\infty$ in finite time.  

    Putting this together shows that the only solution to the Ricatti differential inequality (\ref{eqn-Ricat}) that exists for all $t \in (-\infty, \infty)$ is the constant solution $\phi(t) = 0$.  Thus we obtain that $\Delta f=0$ on $M$.  Plugging back into the Bochner formula gives that $\mathrm{Hess} f  = 0$ which implies that $\nabla f$ is parallel and gives the desired splitting. 
\end{proof}

Now we consider the case $\Lambda \neq 0$.  First note that if 
\[ F_{\Lambda}= \frac{1}{2} |\nabla f|^2 - \Lambda f = c\]
then we can add a constant to $f$ in order to make $F_{\Lambda}=0$.  We can also assume that $\Lambda >0$ by taking $-f$ instead of $f$.   From the rest of this section we  assume that our function is normalized in this way, that is we have the equation
\begin{equation} \label{F=0}
    |\nabla f|^2 =  2\Lambda f,  \qquad \Lambda >0.
\end{equation}
Let  $C = \{ x:  \nabla f = 0 \}$ be the critical point set of $f$, by (\ref{F=0}) $C$ is also the zero set of $f$.  Note that we also have from the normalization that $f\geq 0$.

\begin{proposition}
If $(M,g)$ is a Riemannian manifold and $f$ is a function satisfying (\ref{F=0}) then the function $\rho = \sqrt{\frac{2f}{\Lambda}}$ defined on $M \setminus C$ is a distance function $|\nabla \rho|=1.$
\end{proposition}

\begin{proof}
    \begin{align*}
        \nabla \rho &=  \frac{1}{2}\left(  \sqrt{\frac{2}{\Lambda f }} \right) \nabla f \\ \\ |\nabla \rho|^2  &= \frac{|\nabla f|^2}{2 \Lambda f} \\ \\
        &=  1.
         \end{align*}
\end{proof}

Now let $S_f$ be the $(1,1)$-tensor $S_f(X) = \nabla_X \nabla f$.  We have the following lemma about $S_f$ for a function that satisfies (\ref{F=0}). 

\begin{lemma} \label{radial-eqn}
Let $(M,g)$ be a Riemannian manifold and $f$ a function satisfying (\ref{F=0}) then 
\begin{align*}
S_f(\nabla f ) &= \Lambda \nabla f, \\ 
R(X, \nabla f) \nabla f &= -(\nabla_{\nabla f}S_f)(X) - S_f\circ(S_f-\Lambda I)(X).
\end{align*}
\end{lemma}
\begin{proof}
The first equation follows from (\ref{F=0}) since $S_f(\nabla f) = \nabla_{\nabla f} \nabla f = \frac{1}{2} \nabla |\nabla f|^2$. For the second equation we compute
\begin{align*} 
R(X,\nabla f)\nabla f & =\nabla^2_{X,\nabla f}\nabla f-\nabla^2_{\nabla f,X}\nabla f \\  &= \left(\nabla_XS_f\right)(\nabla f) - \left(\nabla_{\nabla f}S_f\right)(X)\\
&= \nabla_X (S_f(\nabla f)) - S_f(\nabla_X \nabla f) - \left(\nabla_{\nabla f}S_f\right)(X)\\
&= \Lambda S_f(X) - S_f \circ S_f(X) - \left(\nabla_{\nabla f}S_f\right)(X)\\
&= - S_f\circ(S_f-\Lambda I)(X) - \left(\nabla_{\nabla f}S_f\right)(X).
\end{align*}
\end{proof}

We now also consider the function $\rho$, and let $S_{\rho}$ be the $(1,1)$-tensor $S_{\rho}(X) = \nabla_{X} \nabla \rho$. Geometrically, $S_{\rho}$ is the shape operator of the level sets.   Since $\rho$ satisfies $|\nabla \rho | = 1$, $S_{\rho}$ satisfies the radial curvature equations
\begin{align*}
S_{\rho}( \nabla \rho) & = 0, \\
R(X, \nabla \rho) \nabla \rho &= - \left( \nabla_{\nabla \rho} S_{\rho}\right)(X) - S_{\rho} \circ S_{\rho}(X).
\end{align*}

Also, from $|\nabla \rho|=1$ it follows that the integral curves of $\rho$ are unit speed geodesics. Using standard techniques in Riemannian geometry, we have the following Ricatti equation for $S_{\rho}$ along its integral curve. 

\begin{lemma} \label{lem-ricatti}
Let $\gamma(\rho)$ be a unit speed geodesic such that $\gamma' = \nabla \rho$ and $R(X, \nabla \rho) \nabla \rho = 0$.  Let $E(\rho)$ be a unit length perpendicular parallel field along $\gamma$  that is an eigenvector for $S_{\rho}$ at some point $\gamma(\rho_0)$ and let $\phi(\rho)= g(S_{\rho}(E(\rho)), E(\rho)).$  Then $\phi$ satisfies the equation 
\[ \phi' = - \phi^2\]
and $E(\rho)$ is a eigenvector for $S_{\rho}$ for every point on the geodesic $\gamma$. 
\end{lemma}

\begin{proof}
While this proof uses standard techniques, we outline the proof for completeness.

First note that, by the radial curvature equation, we directly obtain 
\[ \phi' = \frac{d}{d\rho} g(S_\rho(E), E) = -|S(E)|^2 \leq -(g(S(E), E))^2  = - \phi^2,  \]
where the only place where we have used the inequality is in applying Cauchy-Schwarz.  Despite this, obtaining the opposite inequality takes a circuitous route. 

Now consider a vector field $J$ along $\gamma$ such that $[ \gamma', J]=0$.  This is equivalent to letting $J$ solve $J' = S_{\rho}(J)$, where $J'$ denotes covariant derivative along $\gamma.$

Now, $J$ won't be constant length so we need to consider the normalized function 
\[ \psi = \frac{ g(S_{\rho}(J), J )}{|J|^2}.  \]
Then by a calculation assuming $R(X, \nabla \rho) \nabla \rho = 0$  we have that $\psi$ satisfies
\[ \psi ' \geq - \psi^2, \]
where again the inequality comes from applying Cauchy-Schwarz.  

Now the lemma follows from observing that $\psi=\phi$.  To see this, note that taking a derivative of the defining equation for $J$ gives that $J$ is a Jacobi field along $\gamma$, as  
\[ J'' = \nabla_{\nabla \rho} (S_\rho(J)) = \nabla_{\nabla \rho} \nabla_J \nabla \rho = - R(J, \nabla \rho)\nabla \rho .\]
 On the other hand, $R(J ,\nabla \rho) \nabla \rho = 0$, implies that the Jacobi fields along $\gamma$ are of the form $J(\rho) = F + \rho E$ for some parallel fields $E$ and $F$.  Then,  if we choose $J$ to be the solution to  $J' = S_{\rho}(J)$ such that $J(\rho_0) = E(\rho_0)$ and $E(\rho_0)$ is an eigenvector of $S$, we obtain that $J$ is the Jacobi field with $J(\rho_0) = E(\rho_0)$ and $J'(\rho_0) = \Lambda_0 E (\rho_0)$, where $\Lambda_0$ is the corresponding eigenvalue.  In particular, the form of the Jacobi fields implies that $J(\rho)$ is a linear function times $E(\rho)$ for all $\rho$, by definition this implies $\phi = \psi$. 

 Now the Cauchy-Schwarz inequality above must be an equality, which shows that $E(\rho)$ is an eigenvector along all of $\gamma$.
\end{proof}
This equation now allows us to use a radial sectional curvature bound to control the Hessian of $f$. We are now ready to prove the main theorem of this section. 

\begin{proof}[Proof of Theorem \ref{thm-rigid}]

Proposition \ref{prop-rigid-0} implies the result when $\Lambda=0$.  When $\Lambda \neq 0$ we normalize so that we have (\ref{F=0}). 

On the minimum set, $C$, Lemma \ref{radial-eqn} shows that the only eigenvalues of $S_f$ are $0$ and $\Lambda$.  For points not on $C$, we can argue along a geodesic $\gamma$ with $\gamma'=\nabla \rho.$  Let $\phi_i(\rho)$ be the $i$-th eigenvalue of $S_{\rho}$ at $\gamma(\rho)$. Note that by Lemma \ref{radial-eqn} $\nabla f$ is an eigenvector of $S_f$ with eigenvalues $\Lambda$, so we only need consider vectors perpendicular to $\nabla f$.   By Lemma \ref{lem-ricatti}, $\phi_i' = -\phi_i^2$.  The solutions to this ODE are 
\[ \phi_i = \frac{1}{\rho-a_i} \quad \text{or} \quad \phi_i = 0\]
for a constant $a_i$. Since we require that $\phi_i$ exists for all $\rho>0$, this implies that $a_i<0$ for all $i$. 

On the other hand, from $f= \frac{\Lambda}{2} \rho^2$ we have that
\[ \mathrm{Hess} f = \Lambda d\rho \otimes d\rho + \Lambda \rho \mathrm{Hess} \rho.\]
So the eigenvalues of $S_f$ along $\gamma$ are of the form, 
\[ \frac{\Lambda \rho}{\rho-a_i} \quad \text{or} \quad 0. \]
Since $\rho>0$ and $a_i\leq 0$, the eigenvalues of $S_f$ are bounded between $0$ and $\Lambda$.  Moreover, if $a_i\neq 0$, the function $\frac{\Lambda \rho}{\rho-a_i}$ is $0$ at $\rho=0$ and increases to $\Lambda$ as $\rho \to \infty$. This would contradict that $\Delta f$ is constant. 

Therefore, we must have that the only eigenvalues of $S_f$ on all of $M$ are $0$ and $\Lambda$. In particular, $\mathrm{Hess} f \geq 0$ and $C$ is a totally convex set as if $p, q \in C$ and $\gamma$ is a geodesic connecting $p$ and $q$ then $\gamma$ must also lie on $C$. In particular, $C$ is connected.

$C$ is also the zero set of the map $p \to \nabla f(p)$, using normal coordinates, we can also compute that the kernel of this map is the null space of $S$.  Therefore, by the constant rank theorem, $C$ is a submanifold whose tangent space is the null space of $S$.  Then since $C$ is a totally convex set, it is a totally geodesic submanifold, and its second fundamental form vanishes. 

Now consider the normal exponential map to $C$,  $exp^{\perp}: \nu(C) \to M$.  Since $\rho$ is smooth when $\rho>0$, this map is a diffeomorphism.  The metric is then determined by the second fundamental form of $C$ and the Jacobi fields along integral curves of $\rho$.  Since the second fundamental form is zero and the assumption  $R(X, \nabla \rho)\nabla \rho=0$ determines the Jacobi fields,  we can see that the metric is isometric to a flat normal bundle. \end{proof}

\noindent

\section{Bounds of trace and rigidity of $q$-solitons}

Our first key tool is a Hamilton type identity which it was obtained by Hamilton \cite{Hamilton} for gradient Ricci solitons.

\begin{proposition}\label{Hamilton identity}  A complete gradient $q$-soliton satisfies \eqref{equality assumption} if, and only if,  \begin{equation}\label{IH}
|\nabla f|^2 - \frac{1}{2}\operatorname{tr}(q)-2\lambda f=C,
\end{equation}
where $C$ is a constant.
    
\end{proposition}
\begin{proof}
Since we assume \eqref{equality assumption}, then it follows Hamilton's idea to have 
\begin{align*}
 \nabla_{i}\left( |\nabla f|^2 - \frac{1}{2}\operatorname{tr}(q)-2\lambda f  \right)&=2(\nabla_{i}\nabla_{j}f-\frac{1}{2} q_{ij}-\lambda g_{ij})\nabla_{j}f\\
&=0.
\end{align*}
Conversely, if holds \eqref{IH} then from soliton equation \eqref{grad q soliton}
\begin{eqnarray*}
    0 &=& \nabla_{i}\left( |\nabla f|^2 - \frac{1}{2}\operatorname{tr}(q)-2\lambda f  \right)\\
    &=& 2\nabla_{\nabla f}\nabla f-\frac{1}{2}\nabla\operatorname{tr}(q)-2\lambda \nabla f\\
    &=& Q(\nabla f)-\frac{1}{2}\nabla\operatorname{tr}(q).
\end{eqnarray*}
This conclude the proof.
\end{proof}

The second key lemma  provides a formula to the Laplacian of the trace of a gradient $q$-soliton to the $q$-flow. In particular it generalizes the classical Laplacian of the scalar curvature of gradient Ricci solitons.
\begin{lemma}\label{Laplacian of the trace}
Let $(M,g,f)$ be a complete gradient $q$-soliton satisfying \eqref{equality assumption}. Then
\begin{equation}\label{eq with trace}
\Delta \operatorname{tr}(q)=2\lambda \operatorname{tr}(q)+|q|^2+2(\mathrm{div}(q))(\nabla f),
\end{equation}
and
\begin{equation}\label{eq f-Laplac of trace}
\Delta_f \operatorname{tr}(q)=2\lambda\operatorname{tr}(q)+|q|^2+\langle 2\mathrm{div}(Q)-\nabla \operatorname{tr}(q),\nabla f\rangle.
\end{equation}
In particular, if $q$ is trace-free, then
    \begin{equation}\label{eq no trace}
        (\mathrm{div}(q))(\nabla f)+\frac{1}{2}|q|^2=0.
    \end{equation}
\end{lemma}
\begin{proof}
From Hamilton's identity \eqref{IH} we have 
    $$\frac{1}{4}\Delta \operatorname{tr}(q)=-\lambda\Delta f+\frac{1}{2}\Delta|\nabla f|^2.$$
 Using Bochner's formula and \eqref{equality assumption} we get
 \begin{eqnarray}\label{eq}
\nonumber\frac{1}{4}\Delta \operatorname{tr}(q) &=& -\lambda\Delta f+{\rm Ric}(\nabla f,\nabla f)+\langle\nabla f,\nabla\Delta f\rangle+|{\rm Hess}(f)|^2\\\nonumber
&=& -\lambda\Delta f+{\rm Ric}(\nabla f,\nabla f)+\frac{1}{2}\langle\nabla f,\nabla \operatorname{tr}(q)\rangle+|{\rm Hess}(f)|^2\\
&=& -\lambda\Delta f+{\rm Ric}(\nabla f,\nabla f)+q(\nabla f,\nabla f)+|{\rm Hess}(f)|^2.\\\nonumber
 \end{eqnarray}
Substituting  
$$|{\rm Hess}(f)|^2=\lambda^2n+\lambda \operatorname{tr}(q)+\frac{1}{4}|q|^2$$ in the above equation we obtain 
 \begin{eqnarray}\label{eq03}
\nonumber\frac{1}{4}\Delta \operatorname{tr}(q) &=& -\lambda ^2n-\frac{\lambda}{2}\operatorname{tr}(q)+{\rm Ric}(\nabla f,\nabla f)+q(\nabla f,\nabla f)+\lambda^2n+\lambda \operatorname{tr}(q)+\frac{1}{4}|q|^2\\\nonumber
&=& \frac{\lambda}{2}\operatorname{tr}(q)+\frac{1}{4}|q|^2+{\rm Ric}(\nabla f,\nabla f)+q(\nabla f,\nabla f).\\\nonumber
\end{eqnarray} 
On the other hand, taking the divergence of the soliton equation, along with  \eqref{equality assumption} gives us
\begin{eqnarray}\label{Ric and div}
  \nonumber(\mathrm{div}(q))(\nabla f)  &=& \langle 2\mathrm{div}({\rm Hess}(f)),\nabla f\rangle\\\nonumber
  &=& 2\langle {\rm Ric}(\nabla f)+\nabla\Delta f,\nabla f\rangle \\
  &=& 2{\rm Ric}(\nabla f,\nabla f)+\langle \nabla \operatorname{tr}(q),\nabla f\rangle\\
  &=& 2{\rm Ric}(\nabla f,\nabla f)+2q(\nabla f,\nabla f).\nonumber
\end{eqnarray}
Substituting this in the above equation we conclude the proof.
\end{proof}

We can also rewrite equation \eqref{eq f-Laplac of trace} as
\begin{equation}
    \Delta_f\operatorname{tr}(q)=\left|-q+\frac{1}{n}\operatorname{tr}(q)g\right|^2+\operatorname{tr}(q)\left(1+\frac{1}{n}\operatorname{tr}(q)\right)+\langle 2\mathrm{div}(Q)-\nabla \operatorname{tr}(q),\nabla f\rangle,
\end{equation}

\begin{remark}
    If we take $q=-2{\rm Ric}$ in the equation \eqref{eq f-Laplac of trace}, then we obtain the well known formula by to gradient Ricci solitons
    $$\Delta_fR=2\lambda R-2|{\rm Ric}|^2.$$ See \cite{WP}.
    \end{remark}

Now we can use these formulae to prove some of the results mentioned in the introduction. 
\eqref{equality assumption} we can prove Theorem \ref{thm01}. 
\begin{proof}[Proof of Theorem \ref{thm01}]
Tracing the soliton equation we have
    $$\Delta f=\lambda n+\frac{1}{2}\operatorname{tr}(q),$$
    and from Hamilton's identity we obtain that
    \begin{equation}
        \Delta f-|\nabla f|^2+2\lambda f=\lambda n-C.
    \end{equation}
    Thus, if $M$ is compact and $\lambda\leq0$ we can see that $f$ is a constant. 
\end{proof}

As a second consequence we can prove Theorem \ref{thm03}.
\begin{proof}[Proof of Theorem \ref{thm03}]
   Since $M$ is compact we have from the divergence theorem and equations \eqref{eq} and \eqref{eq03} that
$$0=\frac{1}{2}\int_M(\mathrm{div}(q))(\nabla f)+\int_M|{\rm Hess}f|^2.$$
Since $M$ is non-stationary, we conclude the proof.
\end{proof}

Before we prove the next result let us recall that a Riemannian manifold is said to be  {\em parabolic} if the unique subharmonic function $u$ on $M$ such that $\sup_Mu<\infty$ must be the constant functions (see \cite{grigoryan} for more details).  Now we can state the proof of Theorem \ref{thm02}. 

\begin{proof}[Proof of Theorem \ref{thm02}]
If $q$ is trace-free and divergence-free, then equation \eqref{eq no trace} implies the space is $q$-flat. If (ii) holds, we have from equation \eqref{Ric and div} that $(\mathrm{div}(q))(\nabla f) \geq0$, and again from \eqref{eq no trace} we obtain $q$-flatness. (iii) follows from equations \eqref{eq with trace} and \eqref{Ric and div}. Finally, if holds item (iv), we have from equation  \eqref{eq with trace} that
\begin{equation}\label{eq12}
    \frac{1}{2}\Delta\operatorname{tr}(q)=2\lambda\operatorname{tr}(q)+|q|^2\geq2\lambda\operatorname{tr}(q)\geq0,
\end{equation}
 since $\lambda<0$ and $\operatorname{tr}(q)\leq0$. Then $\operatorname{tr}(q)$ is a subharmonic function, and being $M$ parabolic we obtain constant trace, that is, from \eqref{eq12} we have $q$ is trace-free and $M$ must be $q$-flat.

On the other hand, if  $\text{Ric}\geq0$ we consider the cut-off function $\phi_r\in C_0^\infty(B(x_0,2r))$ for $r>0$, such that
\begin{eqnarray}\label{prob 0}\left\{ \begin{array}{lllll}
0\leq\phi_r\leq1 & {\rm in}\,\,B(x_0,2r)\\
\,\,\,\,\,\phi_r=1 & {\rm in}\,\, B(x_0,r)\\
\,\,\,\,\,|\nabla\phi_r|^2\leq\dfrac{C}{r^2} & {\rm in}\,\,B(x_0,2r)\\
\,\,\,\,\,\Delta\phi_r\leq\dfrac{C}{r^2} & {\rm in}\,\,B(x_0,2r),\\
\end{array}\right.
\end{eqnarray}
where $C>0$ is a constant. Its existence is well known for manifolds with nonnegative Ricci curvature(see \cite[Theorem 2.2]{Batu}). Thus, from equation \eqref{eq12} we have
$$0\leq\int_{B_{2r}}2\lambda\operatorname{tr}(q)\phi_r\leq\frac{1}{2}\int_{B_{2r}}\phi_r\Delta\operatorname{tr}(q)=\frac{1}{2}\int_{B_{2r}}\operatorname{tr}(q)\Delta\phi_r\leq\frac{1}{2}\int_{{B_{2r}}\diagdown B_r}\frac{C}{r^2}\operatorname{tr}(q)\leq0.$$
Therefore, $\operatorname{tr}(q)=0$, and from \eqref{eq12} we conclude $q$-flat.

\end{proof}

Observe that a result with a slightly different assumption than  (i) of Theorem \ref{thm02} has been considered by first author and Griffin \cite[Theorem 2.6]{Cunha and Griffin}. In the same work the authors showed many other results by considering non-positive Ricci curvature for complete gradient $q$-solitons, while the itens (ii) and (iii) in the we have been considered non-negative Ricci curvature.

 Lemma \ref{Laplacian of the trace} shows that any steady gradient $q$-soliton such that $q$ has constant trace must be $q$-flat provided it satisfies the condition $(\mathrm{div}(q))(\nabla f) \geq0$. The next result give us an estimate of the trace of $q$ in both cases expanding and shrinking.

 \begin{proposition}\label{Prop01}
 Consider a gradient $q$-soliton satisfying   
 \eqref{equality assumption}  such that $q$ is constant trace and  $(\mathrm{div}(q))(\nabla f) \geq0$, $\lambda\not=0$. If $\lambda>0$ then $-2\lambda n\leq \operatorname{tr}(q)\leq0$, and if $\lambda<0$ then $0\leq \operatorname{tr}(q)\leq-2\lambda n$. Furthermore, if $\operatorname{tr}(q)=0$ then $M$ must be $q$-flat, and if $\operatorname{tr}(q)=-2\lambda n$ then $q=-2\lambda g$.
 \end{proposition}
 \begin{proof}
 From \eqref{eq with trace} we have
     $$2\lambda\operatorname{tr}(q)+|q|^2\leq0.$$
 So, by the Cauchy-Schwarz inequality
 we have 
 $$-2\lambda\operatorname{tr}(q)\geq|q|^2\geq\frac{1}{n}(\operatorname{tr}(q))^2,$$
 i.e.,
 $$\operatorname{tr}(q)^2+2\lambda n\operatorname{tr}(q)\leq0.$$
This imply that, $-2\lambda n\leq \operatorname{tr}(q)\leq0$ if $\lambda>0$ and $0\leq \operatorname{tr}(q)\leq-2\lambda n$ if $\lambda<0$.

 Finally, if $\operatorname{tr}(q)=-2\lambda n$, then $$q=\frac{\operatorname{tr}(q)}{n}g=-2\lambda g.$$
 The case $\operatorname{tr}(q)=0$ is trivial.
 \end{proof}

 In particular, if $q=-2{\rm Ric}$, we obtain Propositon 3.3 of \cite{WP} for gradient Ricci solitons. 

\begin{corollary}[Wylie-Petersen]
 Let $(M,g,f)$ be a gradient Ricci soliton with constant scalar curvature and $\lambda\not=0$. If $\lambda>0$ then $0\leq R \leq\lambda n$ and if $\lambda<0$ then $\lambda n\leq R \leq0$. In either case the metric is Einstein when the scalar curvature equals either of the extreme values.
 \end{corollary}

 We also have a weaker condition than constant trace to prove the following result on the sign of the trace $\operatorname{tr}(q)$
of a gradient $q$-soliton depending on some asymptotic behavior of it (a similar result was obtained in \cite[Theorem 1]{ML} for gradient Yamabe solitons).

\begin{proposition}
 Let $(M, g, f)$ be a complete and non-compact gradient $q$-soliton with $\lambda\geq0$. Assume that $\liminf_{x\to\infty}(-\frac{1}{2}\operatorname{tr}(q)(x))\geq0$ and $(\mathrm{div}(q))(\nabla f) \geq0$. Then the trace of $q$ is non-positive.
\end{proposition}
\begin{proof}
    Consider the function $G(x)=-\frac{1}{2}\operatorname{tr}(q)(x)$. Assume that $\inf_M G(x)<0$. Since $\liminf_{x\to\infty}G(x)\geq0$ we know that there is some point $x_0$ such that $G(x_0)=\inf_MG(x)<0$. Hence 
    $$\Delta G(x_0)\geq0,\,\,\,\nabla G(x_0)=0.$$
  So that, in this point $x_0$, from equation \eqref{eq with trace} we obtain
  $$0\geq\Delta\operatorname{tr}(q)\geq 2\lambda\operatorname{tr}(q)+|q|^2\geq2\lambda\operatorname{tr}(q).$$
  This is an absurd since $\lambda\operatorname{tr}(q)(x_0)>0$ for $\lambda\geq0$. Therefore, the strong maximum principle implies that either $G(x)>0$ or $G(x)=0$ on $M$.
\end{proof}

We finish this section by also  recalling the variation of scalar curvature of a $q$-flow which is similar to the equations above (see \cite[Lemma 2.8]{B Chow}).
\begin{lemma}
Let $(M,g(t))$ be a one-parameter
family of metrics satisfying a $q$-flow equation. Then
\begin{equation}\label{evaluation of R}
    \frac{\partial}{\partial t}R=-\Delta \operatorname{tr}(q)+\mathrm{div}(\mathrm{div}(q))-\langle q, {\rm Ric}\rangle.
\end{equation}
\end{lemma}
\begin{remark}
{\em 
   Observe that if $X=\nabla f$ is a conformal vector field in a gradient $q$-soliton, then from \cite{Cunha and Griffin} we have $n \mathrm{div}(q)=\nabla \operatorname{tr}(q)$. That is, $n \mathrm{div}(\mathrm{div}(q))=\Delta\operatorname{tr}(q)$. Therefore, from above lemma we obtain
$$\frac{\partial}{\partial t}R=-\frac{n-1}{n}\Delta \operatorname{tr}(q)-\langle q, {\rm Ric}\rangle.$$ }
\end{remark}

Also, observe that if we take $q=-2{\rm Ric}$ in \eqref{evaluation of R} and using the contracted second Bianchi
identity we get the known evolution equation for the scalar curvature of the Ricci flow
$$\frac{\partial}{\partial t}R=\Delta R-2|{\rm Ric}|^2.$$

When we consider a self-similar gradient $q$-soliton solution to the $q$-flow,  we can state the evolution equation of the trace of the tensor $q$.
\begin{lemma}\label{evolution trace}
  If $(M,g,f)$  is a gradient $q$-soliton to the $q$-flow, then 
\begin{equation}\label{evolution of trace}
\frac{\partial}{\partial t}\operatorname{tr}(q)=-|q|^2-2\lambda\operatorname{tr}(q)-g^{kl}\left[2(\mathrm{div}(q))_l-\nabla_l\operatorname{tr}(q)\right]\nabla_kf.
\end{equation}
\begin{proof}
From soliton equation we have
\begin{equation}\label{eq 5}
\frac{1}{2}\frac{\partial}{\partial t}\operatorname{tr}(q)=\frac{\partial}{\partial t}\Delta f=\frac{\partial}{\partial t}(g^{ij}\nabla_i\nabla_jf)=(\frac{\partial}{\partial t}g^{ij})\nabla_i\nabla_jf+g^{ij}\frac{\partial}{\partial t}(\nabla_i\nabla_jf).
\end{equation}
Since $\frac{\partial}{\partial t}g^{ij}=-g^{ik}g^{jl}\frac{\partial}{\partial t}g_{kl}$, we have $\frac{\partial}{\partial t}g^{ij}=-q_{ij}$. The second term on the right is given by
$$g^{ij}\frac{\partial}{\partial t}(\nabla_i\nabla_jf)=g^{ij}\frac{\partial}{\partial t}\left(\frac{\partial^2f}{\partial x^i\partial x^j}-\Gamma_{ij}^k\nabla_kf\right)=-g^{ij}\left(\frac{\partial}{\partial t}\Gamma_{ij}^k\right)\nabla_kf.$$
 On the other hand it is known that the variation of Christoffel symbols (see \cite{B Chow}) is given by
$$\frac{\partial}{\partial t}\Gamma_{ij}^k=\frac{1}{2}g^{kl}\left(\nabla_iq_{jl}+\nabla_jq_{il}-\nabla_lq_{ij}\right)$$
    for any a one-parameter
family metric satisfying $q$-flow equation. Therefore equation \eqref{eq 5} together soliton equation becomes
\begin{eqnarray*}
   \frac{1}{2}\frac{\partial}{\partial t}\operatorname{tr}(q) &=&-q_{ij}\cdot\nabla_i\nabla_jf-\frac{1}{2}g^{ij}g^{kl}\left(\nabla_iq_{jl}+\nabla_jq_{il}-\nabla_lq_{ij}\right)\nabla_kf\\
   &=& -\frac{1}{2}|q|^2-\lambda \operatorname{tr}(q)-\frac{1}{2}g^{ij}g^{kl}\left(\nabla_iq_{jl}+\nabla_jq_{il}-\nabla_lq_{ij}\right)\nabla_kf.
\end{eqnarray*}
We conclude the proof using that $(\mathrm{div}(q))_l=g^{ij}\nabla_iq_{jl}$.
\end{proof}
\end{lemma}

\begin{remark} Note that, from equations \eqref{evolution of trace} and \eqref{eq f-Laplac of trace} we can also  rewrite the evolution of the trace as
\begin{equation}
    \frac{\partial}{\partial t}\operatorname{tr}(q)=-\Delta_f\operatorname{tr}(q).
\end{equation}
\end{remark}

\begin{remark}
{\em 
 Taking $q=-2{\rm Ric}$ in the above lemma we obtain the evolution of the scalar curvature of a gradient Ricci soliton.
 $$\frac{\partial}{\partial t}R=-2\lambda R+2|{\rm Ric}|^2.$$}
\end{remark}

\section{Bounds of potential function and volume growth estimate}

We start this section by investigating  an upper estimate for the potential function. We assume \eqref{IH}
$$|\nabla f|^2 - \frac{1}{2}\operatorname{tr}(q)-2\lambda f=C.$$
As above, let  $F=\frac{1}{2}|\nabla f|^2-\lambda f$, and observe that $F$ is bounded above if and only if $\operatorname{tr}(q)$ is bounded above.  By adding a constant to $f$  we can assume that
$$\frac{1}{2}|\nabla f|^2-\lambda f\leq 0.$$
From now on,  we will consider this normalization on $f$. Now we have the following Lemma. 

\begin{lemma}\label{lema estimates}
Let $\left(M,g\right)$ be a complete Riemannian manifold such that 
$F\leq  C$ for some smooth function $f$ and $\lambda>0$.  Then the function $f$ satisfies the following:
\begin{align} \label{f estimate}
f(x)\leq & \frac{\lambda}{2}\left(r(x)+2\sqrt{\frac{f(x_0)}{2\lambda}}\right)^2, \\ \label{grad estimate}
|\nabla f|(x) \leq& \lambda r(x) + \sqrt{2\lambda f(x_0)}.
\end{align}
Here $r(x)=d(x,x_0)$ is the distance function from some fixed point $x_0 \in M$.
\end{lemma}
\begin{proof}
Since $F\leq C$ we have that
$$0\leq|\nabla f|^2\leq 2\lambda f,$$
or
$$|\nabla\sqrt{f}|\leq \frac{\sqrt{2\lambda}}{2},$$
in particular  $f>0$. This is the only fact about $f$ used in the proof of  \cite{Cao-Zhou} to  obtain \eqref{f estimate}. 

Since $|\nabla f|^2\leq 2\lambda f$ we get \eqref{grad estimate} using \eqref{f estimate}. 
\end{proof}

Observe that inequality \eqref{f estimate} provides  the estimate of Theorem \ref{thm04}. The inequality \eqref{bound from below} follows from the following Proposition

\begin{proposition}\label{Prop03}
Let $(M,g,f)$ be a complete noncompact shrinking  gradient $q$-soliton such that $F\leq C$. If $q$ satisfies the inequality \eqref{Integral ineq}, then the potential function $f$ satisfies the estimate 
$$
f(x) \geq \frac{\lambda}{2}\left(r(x)-c_1\right)^2,$$
where $c_1$ is a positive constant depending only on $n$ and the tensor  $q$ on the unit ball $B_{x_0}(1)$.
\end{proposition}
\begin{proof}
Following the ideas of \cite{Cao-Zhou}, from Remark \ref{Rmk01} let us consider any minimizing normal geodesic $\gamma(s)$ as the same properties and define $\phi(s)$ by
\begin{equation*}
    \phi(s)= \left\lbrace\begin{matrix}
      s,  & s \in \left[0,1\right] ,     \\
      1,  & s\in \left[1,s_0-1\right],          \\
       s_0-s, &  s\in \left[s_0-1,s_0\right]  .        
    \end{matrix}\right.
\end{equation*}
Integrating the soliton equation along $\gamma$ from $s=1$ to $s=s_0-1$ and using the hypothesis \eqref{Integral ineq} we get the lower estimate \begin{align*}
f^{\prime}(\gamma(s_0-1))-f^{\prime}(\gamma(1)) &= \int_{1}^{s_0-1} \nabla_{X}f^{\prime}(\gamma(s))ds\\
 &=\lambda (s_{0}-2) + \int_{1}^{s_0-1} \phi^2(s)\left( \frac{1}{2}q(X,X) \right)ds\\
 &= \lambda (s_{0}-2) + \frac{1}{2}\int_{0}^{s_0} \phi^2(s)q(X,X)ds-\frac{1}{2}\int_0^1\phi^2(s)q(X,X)ds\\
 &- \frac{1}{2}\int_{s_0-1}^{s_0}\phi^2(s)q(X,X)ds\\
 &\geq \lambda(s_{0}-2) -(n-1)\int_{0}^{s_0} |\phi'(s)|^2ds-\max_{B_{x_0}(1)}\left|-\frac{1}{2}q\right|\\
 &- \frac{1}{2}\int_{s_0-1}^{s_0}\phi^2(s)q(X,X)ds\\
 &= \lambda(s_{0}-2) -2(n-1)-\max_{B_{x_0}(1)}\left|-\frac{1}{2}q\right|- \frac{1}{2}\int_{s_0-1}^{s_0}\phi^2(s)q(X,X)ds,\\
     &\\
\end{align*}
where in the second equality we used that
\begin{equation}\label{eq06}
    \nabla_Xf'=\nabla_X\nabla_X f=\lambda+\frac{1}{2}q(X,X).
\end{equation}
Using again \eqref{eq06} and integrating by parts the last term on the right of the above inequality to have
\begin{align*}
\int_{s_0-1}^{s_0} \phi^2(s)\left(-\frac{1}{2}q(X,X)\right)ds  &=  \lambda\int_{s_0-1}^{s_0}\phi^2(s)ds - \int_{s_0-1}^{s_0}\phi^2(s)\nabla_{X}f^{\prime}(\gamma(s))ds \\
&= \frac{\lambda}{3} -(s_0-s)^2f'(\gamma(s))|_{s_0-1}^{s_0} +2\int_{s_0-1}^{s_0}\phi(s)\nabla\phi(s)f^{\prime}(\gamma(s))ds.\\
&= \frac{\lambda}{3} + f^{\prime}(\gamma(s_0-1)) -2\int_{s_0-1}^{s_0}\phi(s)f^{\prime}(\gamma(s))ds.
\end{align*}
Thus, substituting this in the above expression we obtain
\begin{equation}\label{eq07}
2\int_{s_0-1}^{s_0}\phi(s)f^{\prime}(\gamma(s))ds \geq \lambda(s_{0}-2) -2(n-1) - \max_{B_{x_0}(1)}\left|-\frac{1}{2}q\right| + \frac{\lambda}{3} + f'(\gamma(1)).
\end{equation}

Next from Lemma \ref{lema estimates} we have seen that
$$|\nabla f|^2\leq 2\lambda f\,\,{\text{and}}\,\,|\nabla\sqrt{f}|\leq \frac{\sqrt{2\lambda}}{2}.$$
This implies that
$$|f'(\gamma(s))|\leq|\nabla f(\gamma(s))|\leq\sqrt{2\lambda f(\gamma(s))},$$
and 
$$|\sqrt{f(\gamma(s))}-\sqrt{f(\gamma(s_0))}|\leq\frac{\sqrt{2\lambda}}{2}(s_0-s)\leq\frac{\sqrt{2\lambda}}{2},$$
in the interval $[s_0-1,s_0]$. Therefore, in this interval
$$\max|f'(\gamma(s))|\leq\sqrt{2\lambda f(\gamma(s))}\leq\sqrt{2\lambda f(\gamma(s_0))}+\lambda.$$

Finally, observe that
$$2\int_{s_0-1}^{s_0}\phi(s)f^{\prime}(\gamma(s))ds\leq\max_{[s_0-1,s_0]}|f'(\gamma(s))|,$$
so that from equation \eqref{eq07}
$$\sqrt{2\lambda f(\gamma(s_0))}\geq \lambda(s_{0}-2) -2(n-1) - \max_{B_{x_0}(1)}\left|-\frac{1}{2}q\right| + \frac{2\lambda}{3} + f'(\gamma(1))=\lambda(s_0-c_1),$$
where $c_1$ depends only of $n$ and the tensor $q$ in the ball $B_{x_0}(1)$. That is,
$$\sqrt{2\lambda f(\gamma(s))}\geq \sqrt{2\lambda f(\gamma(s_0))}-\lambda\geq\lambda(s_0-c_1).$$
This conclude the proof.
\end{proof}

Now let us prove the volume growth estimate of geodesic balls of gradient $q$-solitons. Since $F\leq C$, from Lemma \ref{lema estimates}, we have seen that
$$0\leq|\nabla f|^2\leq 2\lambda f. $$
Define the following function
\begin{equation*}
    \eta^2=\frac{2}{\lambda}f.
\end{equation*}
 It's easy to see that
\begin{equation}\label{eq08}
\nabla \eta=\frac{1}{2}\sqrt{\frac{2}{\lambda}}\frac{\nabla f}{\sqrt{f}}  \ \ \text{and} \ \ |\nabla \eta|= \frac{1}{2}\sqrt{\frac{2}{\lambda}}\frac{|\nabla f|}{\sqrt{f}}\leq 1.
\end{equation}
Moreover, denote the following set and its volume by
$$\Omega(r)=\{x\in M; \eta(x)<r\}\,\,\,{\text{and}}\,\,\,V(r)=\int_{\Omega(r)}dV.$$
The co-area formula gives us
$$V(r)=\int_{0}^{r}ds\int_{\partial \Omega(s)}\frac{1}{|\nabla \eta|}dA,$$
and from \eqref{eq08}
\begin{equation}\label{eq09}
    V^{\prime}(r)=\int_{\partial \Omega(r)}\frac{1}{|\nabla \eta|} dA= \lambda r\int_{\partial \Omega(r)} \frac{1}{|\nabla f|}dA,
\end{equation}
since $\eta=r$ on $\partial\Omega(r)$.

To prove our volume growth estimate we first need the following key lemma. 
\begin{lemma}\label{lemma6} Let be $(M,g,f)$ a gradient shrinking $q$-soliton satisfying \eqref{equality assumption}. Then the following identity holds 
    \begin{equation*}
        nV(r)-rV^{\prime}(r)= -\frac{1}{2\lambda }\int_{\Omega(r)}\operatorname{tr}(q) dV + \frac{1}{2\lambda}\int_{\partial \Omega(r)} \frac{\operatorname{tr}(q)}{|\nabla f|}dV,
    \end{equation*}
\end{lemma}
\begin{proof}
    From soliton equation we have 
    $$ \Delta f = \lambda n + \frac{1}{2} \operatorname{tr}(q).
$$
Integrating this identity on $\Omega(r)$ we obtain \begin{equation}\label{eq4.4}
    \int_{\Omega(r)} \Delta f dV= \lambda nV(r) + \frac{1}{2}\int_{\Omega(r)}\operatorname{tr}(q).
\end{equation}
By Hamilton's type identity \eqref{Hamilton identity}, we have using \eqref{eq08} that \begin{align*}
    \int_{\Omega(r)}\Delta f dV & = \int_{\partial \Omega(r)} \left\langle \nabla f, \frac{\nabla \eta}{|\nabla \eta|} \right\rangle d\sigma \\
    & = \int_{\partial \Omega(r)}|\nabla f| d\sigma \\
    & = \int_{\partial \Omega(r)}\frac{2\lambda f+\frac{1}{2}\operatorname{tr}(q)}{|\nabla f|}d\sigma \\
    & = 2\lambda \int_{\partial \Omega(r)}\frac{f}{|\nabla f|}d\sigma + \frac{1}{2}\int_{\partial \Omega(r)}\frac{\operatorname{tr}(q)}{|\nabla f|} d\sigma\\
    & = \lambda^2r^2\int_{\partial \Omega(r)}\frac{1}{|\nabla f|}d\sigma + \frac{1}{2}\int_{\partial \Omega(r)}\frac{\operatorname{tr}(q)}{|\nabla f|} d\sigma\\
    & = \lambda rV'(r)+\frac{1}{2}\int_{\partial \Omega(r)}\frac{\operatorname{tr}(q)}{|\nabla f|} d\sigma.
\end{align*}
This finishes the proof.
\end{proof}

\begin{remark}\label{remk03}
    Note that the above proof shows that 
    \begin{equation}
-\frac{1}{2}\int_{\Omega(r)}\operatorname{tr}(q)\leq \lambda nV(r).
    \end{equation}
    In fact, from Hamilton's type equality we have
    $$\lambda nV(r) + \frac{1}{2}\int_{\Omega(r)}\operatorname{tr}(q)=\int_{\Omega(r)}\Delta f=\int_{\partial\Omega(r)}|\nabla f|\geq0.$$
\end{remark}

Finally, we can prove Theorem \ref{thm06}.
\begin{proof}[Proof of Theorem \ref{thm06}]
    Consider the function
    $$G(r)=\int_{\Omega(r)}(-\lambda\operatorname{tr}(q))dV,$$
so, by the co-area formula

\begin{equation*}
     G(r)=\int_{0}^{r}ds\int_{\partial \Omega(s)}\frac{-\lambda\operatorname{tr}(q)}{|\nabla \eta|}d\sigma=\lambda\int_{0}^{r}sds\int_{\partial \Omega(s)}\frac{-\lambda\operatorname{tr}(q)}{|\nabla f|}d\sigma,
 \end{equation*}
 so
$$G'(r)=\lambda^2r\int_{\partial\Omega}\frac{-\operatorname{tr}(q)}{|\nabla f|}d\sigma.$$
From above Lemma we can write
\begin{equation}
nV(r)-rV^{\prime}(r)=\frac{1}{2\lambda^2}G(r) - \frac{1}{\lambda^3r}G'(r).
\end{equation}
Thus we obtain
\begin{align*}
     \left(r^{-n}V(r)\right)^{\prime}&= -nr^{-n-1}V(r)+r^{-n}V^{\prime}(r) \\
     &= -\frac{1}{r^{n+1}}(nV(r)-rV'(r))\\
     &= r^{-n-1}\left(-\frac{1}{2\lambda^2}G(r) + \frac{1}{2\lambda^3r}G'(r)\right) \\
     &= -\frac{1}{2\lambda^2}r^{-n-1}G(r) + \frac{1}{2\lambda^3}r^{-n-2}G^{\prime}(r).
 \end{align*}
 By integrating by parts from $r_0$ to $r$ we have from above equation
 \begin{align*}
    r^{-n}V(r)-r_{0}^{-n}V(r_0) &= -\frac{1}{2\lambda^2}\int_{r_0}^{r} r^{-n-1}G(r) + \frac{1}{2\lambda^3}\int_{r_0}^{r} r^{-n-2}G^{\prime}(r)\\
    &= -\frac{1}{2\lambda^2}\int_{r_0}^{r} r^{-n-1}G(r) +\frac{1}{2\lambda^3}r^{-n-2}G(r)\big|_{r_0}^r\\ & - \frac{1}{2\lambda^3}\int_{r_0}^r(-n-2)^{-n-3}G(r)dr \\
    &= \frac{1}{2\lambda^3}r^{-n-2}G(r)-\frac{1}{2\lambda^3}r_{0}^{-n-2}G(r_0) \\ & + \frac{1}{2\lambda^2}\int_{r_0}^{r}r^{-n-3}G(r)\left(\frac{1}{\lambda}(n+2)-r^2\right)dr.
\end{align*}
Now observe that, by definition,  $G$ is positive and increasing in $r$ (since $F\leq C\Leftrightarrow\operatorname{tr}(q)\leq 0$ and $G'(r)>0$), so that, for $r_0=\sqrt{\frac{1}{\lambda}(n+2)}\leq r$ we have $G(r)\geq G(r_0)$ and we obtain the following inequality
\begin{align*}
    \int_{r_0}^{r}r^{-n-3}G(r)\left(\frac{1}{\lambda}(n+2)-r^2\right)dr & \leq G(r_0) \int_{r_0}^{r}r^{-n-3}\left(\frac{1}{\lambda}(n+2)-r^2\right)dr \\
    & = G(r_0)\left( -\frac{1}{\lambda}r^{-n-2}+\frac{r^{-n}}{n}   \right)\bigg|_{r_0}^r \\
    & = G(r_0) \Biggl\{ \left( -\frac{1}{\lambda}r^{-n-2}+\frac{r^{-n}}{n} \right) \\ & - \left( -\frac{1}{\lambda}r_{0}^{-n-2} + \frac{r_{0}^{-n}}{n}  \right)  \Biggr\}. \\
\end{align*}
Substituting this in the above equation we get
$$r^{-n}V(r)-r_{0}^{-n}V(r_0)\leq \frac{1}{2\lambda^3}r^{-n-2}(G(r)-G(r_0)) + \frac{\frac{1}{2\lambda^2}G(r_0)}{n}(r^{-n}-r^{-n}_{0}).$$
Since $r^{-n}\leq r_0^{-n}$ for $r_0=\sqrt{\frac{1}{\lambda}(n+2)}\leq r$, we obtain that
$$r^{-n}V(r)-r_{0}^{-n}V(r_0)\leq \frac{1}{2\lambda^3}r^{-n-2}(G(r)-G(r_0)),$$
so that
$$V(r)\leq (r_{0}^{-n}V(r_0))r^n+\frac{1}{2\lambda^3}r^{-2}G(r).$$
The last term on the right of the above inequality should be from Remark \ref{remk03}
\begin{equation}\label{eq10}
\frac{1}{2\lambda^3}r^{-2}G(r)=-\frac{1}{2\lambda^2}r^{-2}\int_{\Omega(r)}\operatorname{tr}(q)dV\leq\frac{1}{\lambda}nr^{-2}V(r)\leq \frac{1}{2} V(r),
\end{equation}
since $\frac{1}{\lambda}nr^{-2}\leq\frac{n}{n+2}\leq\frac{1}{2}$,
for $r>>1$. Hence,
$$V(r)\leq 2(r_{0}^{-n}V(r_0))r^n.$$
Finally, since $F\leq C$, the definition of function $\eta$ and Theorem \ref{thm05} imply that
\begin{equation}
     \operatorname{Vol}\left( B_{x_0}(r)  \right) \leq V(r+c) \leq V(r_0)r^{n},
 \end{equation}
 for $r$ large enough. This finishes the proof.
\end{proof}

Finally, we still state the following volume estimate.

\begin{proposition}
    Let $(M,g,f)$ be a complete, noncompact  shrinking gradient $q$-soliton satisfying \eqref{equality assumption} and  inequality \eqref{Integral ineq}. Suppose that $F\leq C$ and $q$ satisfies the upper bound \begin{equation}\label{average}
        \frac{1}{V(r)}\int_{\Omega(r)}-\lambda\operatorname{tr}(q) dV \leq \delta,
    \end{equation}
    for some positive constant $\delta <2n\lambda^2$ and every $r$, sufficiently large. Then, there exists some positive constant $C_2>0$ such that \begin{equation*}
    \operatorname{Vol}\left(B_{x_0}(r)\right)\geq C_2 r^{n-\frac{1}{2\lambda^2}\delta},
\end{equation*}
for $r$ sufficiently large.
\end{proposition}
\begin{proof}
    Using \eqref{average} and that $F\leq C$ (i.e. $\operatorname{tr}(q)\leq 0$), we have from Lemma \ref{lemma6} that
    $$
    nV(r)-rV^{\prime}(r)= -\frac{1}{2\lambda}\int_{\Omega(r)}\operatorname{tr}(q) dV +\frac{1}{2\lambda} \int_{\partial \Omega(r)} \frac{\operatorname{tr}(q)}{|\nabla f|}dV\leq\frac{1}{2\lambda^2}\delta V(r),$$
    so that
\begin{equation}
    \left(n-\frac{1}{2\lambda^2}\delta\right)V(r)\leq rV'(r).
\end{equation}
Integrating from $1$ to $r$ we have
$$\ln\frac{V(r)}{V(1)}=\int_1^r(\ln(V(s)))'ds=\int_1^r\frac{V'(s)}{V(s)}ds\geq\int_1^r\frac{n-\frac{1}{2\lambda^2}\delta}{s}ds=\left(n-\frac{1}{2\lambda^2}\delta\right)\ln r.$$
Thus,
$$V(r)\geq V(1)r^{n-\frac{1}{2\lambda^2}\delta}.$$
Finally the definition of function $\eta$ and Proposition \ref{Prop03} imply that
$$\text{Vol}(B_{x_0}(r))\geq V(r-c)\geq 2^{-n}V(1)r^{n-\frac{1}{2\lambda^2}\delta}.$$
\end{proof}

We conclude this section with an application of Proposition \ref{Prop03} by proving the validity of the Omori-Yau maximum principle.

\begin{proof}[Proof of Theorem \ref{thm07}]
    Following the idea of \cite{GR}, we just need to use a sufficient condition for a Riemannian manifold to satisfy the Omori–Yau maximum principle \cite{PRS-Memoirs}. It tell us that in a Riemannian manifold 
$(M,g)$ if there exists a nonnegative $C^2(M)$ function $\psi$
such that
\begin{equation}\label{psi}
    \psi(x) \rightarrow +\infty \,\,\text{if} \ x\to \infty 
\end{equation}
\begin{equation}\label{A}
    \exists A >0\,\,\text{such that} \,\, |\nabla \psi|\leq A \psi^{\frac{1}{2}} \,\, \text{off a compact set, and }
\end{equation}
\begin{equation}\label{B}
    \exists B>0 \,\, \text{such that} \,\, \Delta \psi \leq B \psi^{\frac{1}{2}}G(\psi^{\frac{1}{2}})^{\frac{1}{2}}  \ \text{off a compact set,}
\end{equation}
where $G$ is a smooth function on $[0,+\infty)$ satisfying
\begin{align}\label{eq11}
    \begin{array}{cc}
      i) G(0)>0   & ii) G^{\prime}(t)\geq 0 \ \text{in} \ \left[0, +\infty\right) \\
         iii) G(t)^{-\frac{1}{2}}\notin L^{1}(+\infty) & iv) \limsup_{t\rightarrow +\infty }{\dfrac{tG(t^{\frac{1}{2}})}{G(t)}}<+\infty,
    \end{array}
\end{align}
then the Omori–Yau maximum principle holds on $M$.

Consider the function $G:\left[0, +\infty\right)\rightarrow \mathbb{R}^{+}$ given by $G(t)=t^{2}+1$. It is easy to see that $F$ satisfies the conditions \eqref{eq11}. Finally define  $\psi(x)=f(x)$ on $M$, where $f$ is the potential function. From Proposition \ref{Prop03} we have that  $\psi(x)\geq \frac{1}{4}\left(d(x)-c_1\right)^2$. Then,  $\psi(x)\rightarrow +\infty$ as $x\to\infty$, and $\eqref{psi}$ is satisfied. Also, since $F\leq C$ we saw for $\lambda=\frac{1}{2}$ that $$|\nabla f|^2\leq f,$$
that is, $|\nabla \psi|\leq \sqrt{\psi}$ on $M$, satisfying the condition \eqref{A}. Now, since $\operatorname{tr}(q)\leq0$ we have that \begin{equation*}
     \Delta f= \frac{n}{2} + \frac{1}{2}\operatorname{tr}(q) \leq \frac{n}{2} \leq \frac{1}{4}(d(x)-c_1)^2,
 \end{equation*}
for $d(x)$ large enough. Also, note that $\sqrt{f^2}\leq \sqrt{f}\sqrt{1+f}$ in $M-K$, where $K=\{f<1\}$. That is, 
\begin{equation*}
    \Delta f \leq \frac{1}{4}(d(x)-c_1)^2 \leq f(x) <\sqrt{f^2+1}=\sqrt{f}\sqrt{F(\sqrt{f})},
\end{equation*}
outside a compact set. Thus  \eqref{B} holds, and therefore the Omori–Yau maximum principle
holds for the Laplacian on $M$. 

Finally, to see its validity to the $f$-Laplacian we have proved above that if $u\in C^2(M)$ is a function with $\sup_Mu<+\infty$, there exists a  sequence of points $(x_k)$ on $M$ satisfying i) and iii) in \eqref{eqOY}. Using equation \eqref{Hamilton identity} we saw that 
$$|\nabla f|\leq\sqrt{f+1}=\sqrt{G(\sqrt{f})}.$$
Furthermore, inside of the proof of Theorem 1.9 of \cite{PRS-Memoirs} it is shown that 
$$|\nabla u(x_k)|\leq\frac{C}{k\sqrt{G(\sqrt{f(x_k)})}}.$$
Hence, from above equations we have
\begin{eqnarray*}
  \Delta_fu(x_k) &=& \Delta u(x_k)-\langle\nabla f(x_k),\nabla u(x_k)\rangle\\
  &\leq& \frac{1}{k}+|\nabla f(x_k)||\nabla u(x_k)|\\
  &\leq&\frac{1+C}{k}\to0,\,\,\text{as}\,\,k\to0.
\end{eqnarray*}
This completes the proof.
\end{proof}

\begin{remark}
{\em 
Following the terminology of \cite{PRS}, we say that the $L^1$-Liouville property for $\Delta_f$-superharmonic functions holds if every $Lip_{loc}$ solution of $\Delta_f\leq0$ satisfying $0\leq u\in L^1(M,e^{-f}dvol)$ must be constant. It is shown in \cite[Theorem 24]{PRS} that if the weak maximum principle at infinity (i) and iii) in \eqref{eqOY}) holds for $\Delta_f$ then the $L^1$-Liouville property for $\Delta_f$-superharmonic functions holds. In particular we also have obtained that for any complete gradient shrinking $q$-soliton satisfying \eqref{equality assumption} and \eqref{Integral ineq}, then the $L^1$-Liouville property for $\Delta_f$-superharmonic functions holds.}
\end{remark}

\section*{Statements and Declarations.} The authors state that there is no funding and/or conflicts of interests/competing interests.

\section*{Acknowledgment}
 The first author would like to thank Syracuse University-NY for its hospitality during his visit as a scholar where this work was done. The first author is partially supported by CNPq, Brazil, grant: 301896/2022-4 and 442033/2023-0.

\bibliographystyle{amsplain}

\end{document}